\documentclass[a4paper,10pt]{amsart}

\usepackage{amssymb,amsthm,amsmath,verbatim}
\usepackage{t1enc}

\usepackage{enumerate}
\usepackage{multicol}

\usepackage{color}
\usepackage{tikz}
\usepackage{breqn}
\usepackage{xmpmulti}
\usepackage{psfrag}

 \usepackage[usenames,dvipsnames]{pstricks}
 \usepackage{epsfig}
\usepackage{pst-grad}
\usepackage{pstricks,pst-node}
\usepackage{float}

\usepackage{wrapfig}
\usepackage{framed}

\usepackage[left=3.5 cm, right=3.5cm]{geometry}
\numberwithin{equation}{section}

\newcommand{\force}{{\hspace{0.02 cm}\Vdash}}

\newtheorem{prop}{Proposition}[section]
\newtheorem{dfn}[prop]{Definition}
\newtheorem{lemma}[prop]{Lemma}
\newtheorem{mlemma}[prop]{Main Lemma}

\newtheorem{fact}[prop]{Fact}

\newtheorem{mcor}[prop]{Main Corollary}

\newtheorem{thm}[prop]{Theorem}
\newtheorem{mthm}[prop]{Main Theorem}
\newtheorem{clm}[prop]{Claim}
\newtheorem{mclm}[prop]{Main Claim}

\newtheorem{obs}[prop]{Observation}

\newtheorem{quest}[prop]{Question}

\newcommand{\mc}[1]{\mathcal{#1}}

\newcommand{\oo}{\omega}
\newcommand{\uhr}{\upharpoonright}
\newcommand{\omg}{{\omega_1}}

\DeclareMathOperator{\Fn}{Fn}

\DeclareMathOperator{\htt}{ht}

\def\<{\left\langle}
\def\>{\right\rangle}
\def\br#1;#2;{\bigl[ {#1} \bigr]^ {#2} }

\newcommand{\gen}[2]{\textmd{Gen}(#1,#2)}

\newcommand{\setm}{\setminus}

\newcommand{\subs}{\subset}
\newcommand{\dom}{\operatorname{dom}}

\newcommand{\ran}{\operatorname{ran}}

\newcommand{\vareps}{\varepsilon}

\title[Suslin trees and minimal linear orders]{A model with Suslin trees but no minimal uncountable linear orders other than $\omg$ and $-\omg$}

\date{\today}

  \author{D\'aniel T. Soukup}
  \address[D.T. Soukup]{Universit\"at Wien,
Kurt G\"odel Research Center for Mathematical Logic, Wien, Austria}
 \email[Corresponding author]{daniel.soukup@univie.ac.at}

 \urladdr{http://www.logic.univie.ac.at/$\sim  $soukupd73/}

\makeatletter
\newtheorem*{rep@theorem}{\rep@title}
\newcommand{\newreptheorem}[2]{%
\newenvironment{rep#1}[1]{%
 \def\rep@title{#2 \ref{##1}}%
 \begin{rep@theorem}}%
 {\end{rep@theorem}}}
\makeatother

\newreptheorem{corollary}{Corollary}
\newreptheorem{theorem}{Theorem}

\subjclass[2010]{03E35, 03E04, 06A05}
\keywords{minimal linear order, Suslin tree, Aronszajn tree, rigid, iteration, preservation}

\begin{document}
 \begin{abstract}  
We show that the existence of a Suslin tree does not necessarily imply that there are uncountable minimal linear orders other than $\omg$ and $-\omg$, answering a question of J. Baumgartner.  This is done by a Jensen-type iteration, proving that one can force CH together with a restricted form of ladder system uniformization on trees, all while preserving a rigid Suslin tree.

\end{abstract}
\maketitle

One can quickly see that any infinite linear order either contains a copy of $\oo$ (the order type of the natural numbers) or its reverse $-\oo$.  
In other words, $\pm\omega$ forms a 2-element basis for infinite linear orders. Also, $\oo$ and $-\oo$ are the only \emph{minimal infinite linear orders} in the sense that they embed into each of their infinite suborders.


\begin{wrapfigure}[13]{r}{.5\textwidth}
\centering
 \includegraphics[width=0.4\textwidth]{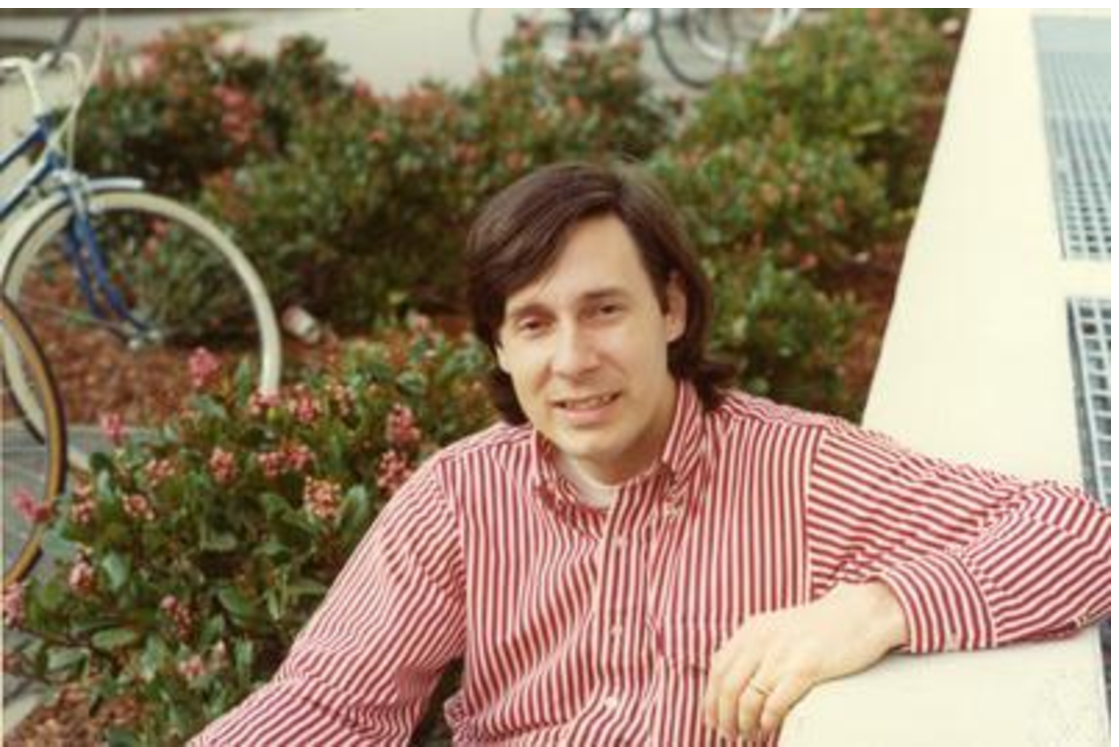}
 
{\small J. Baumgartner in Oberwolfach, 1975\\

(Copyright: George M. Bergman, Berkeley)}
\end{wrapfigure}

Our goal is to address a question of James Baumgartner from the seminal 1982 paper \cite{baumorder} that concerned the analysis and classification of minimal uncountable order types.
Let us first briefly summarize the results that preceded \cite{baumorder}, and the main points of the theory relevant to our paper. The ordinal $\omg$ and its reverse $-\omg$ are \emph{minimal uncountable order types} in the sense that they embed into each of their uncountable suborders. Now, a suborder $L$  of the real line can not contain copies of $\pm \omg$ and, interestingly, there may or may not be minimal uncountable real order types: the Continuum Hypothesis (CH) implies that no uncountable suborder of the reals is minimal \cite{sierpinski1932probleme}, while the Proper Forcing Axiom (PFA) implies that any two $\aleph_1$-dense suborder of the reals are isomorphic \cite{baumgartner1973alphall}. In turn, any $\aleph_1$-dense suborder of the reals is minimal if PFA holds.

There are order types, called \emph{Aronszajn orders}, which embed no copies of $\pm \omg$ nor uncountable suborders of the reals. Baumgartner \cite{baumorder} sketched the construction of a minimal Aronszajn type using $\diamondsuit^+$,\footnote{For a presentation in full detail see \cite[Theorem 18]{stronglysurj}.} and he emphasized multiple (now mostly solved) problems that motivated research in combinatorial set theory in the following 30 years.

Soon after, Stevo Todorcevic \cite{todorcevicpart} introduced the Countryman-line $C(\rho_0)$,\footnote{Recall that a \emph{Countryman line} is an uncountable linear order $C$ so that $C^2$ is the union of countably many chains in the coordinate wise ordering.} which is  minimal under $\textmd{MA}_{\aleph_1}$ \cite[Theorem 2.1.12]{walks}. One culmination of this line of research  came from Justin Moore, who  proved that any Countryman line together with its reverse forms a 2-element basis for Aronszajn linear orders under PFA; in turn, if we add $\pm \omg$ and an $\aleph_1$-dense set of reals, we have a 5-element basis for all uncountable linear orders \cite{moore2006five}.   

Complementing the above results, Moore went on to show that consistently, the only minimal uncountable linear orders are $\pm \omg$ \cite{justinmin}, and this theorem is where our interest lies. Embeddings of Aronszajn linear orders are closely tied to tree embeddings of certain associated Aronszajn trees, and indeed, Moore's result is built on finding a model of CH where a weak version of the ladder system uniformization property holds for trees.

 Baumgartner \cite{baumorder} also asked if, rather than using $\diamondsuit^+$, the existence of a single Suslin tree suffices to find a minimal Aronszajn type. The main result of our paper is a negative answer to this question: we construct a model of CH with a rigid Suslin tree $R$, so that any Aronszajn tree $A$ either embeds a derived subtree of $R$, or $A$ satisfies Moore's uniformization property. All together, we will see that this implies that the only minimal uncountable linear orders are $\pm \omg$ in our model, yet a Suslin-tree exists.

In order to achieve this result, we will apply Ronald Jensen's technique for constructing a ccc forcing iteration of length $\oo_2$ such that each initial segment of this iteration is a Suslin tree itself \cite{devlin2006souslin}. 
Originally, this method was developed to produce a model of CH without any Suslin trees, however, and lucky for us,  Uri Abraham and Saharon Shelah \cite{abrahamiso} developed an analogous iteration theorem which does preserve a fixed Suslin tree $R$ given that each successor step of the iteration preserves $R$ in a strong sense. Let us mention that our framework also provides an alternative way to show Moore's result on minimal linear orders \cite{justinmin}, but using a countably closed preparatory forcing followed by a ccc iteration adding no new reals.

\medskip

First, in Section \ref{sec:prelim}, we will define the various tree-and uniformization properties that we need, and explain in Lemma \ref{lm:base} how an appropriate combination of these can be used to achieve our result on Suslin trees and minimal linear orders. In Section \ref{sec:forceoutline}, we state our Main Theorem \ref{thm:main} that solves the problem of uniformizing a ladder system colouring on a certain tree $A$ while preserving another Suslin tree $R$.  We end this section by showing how to piece together our result (used in successor stages of the iteration), and  the Abraham-Shelah iteration theorem (for limit stages of our forcing) to produce the desired model (see Corollary \ref{cor:main}). Finally, in Section \ref{sec:proof}, we prove our main theorem. We close our paper with a few remarks on our approach and directions for future research.
%

\subsection*{Acknowledgments} We thank S. Friedman and J. Moore for helpful comments.  The author was supported in part by the FWF Grant I1921.

\section{Preliminaries on trees and uniformization}\label{sec:prelim}

\subsection*{Notation} We use fairly standard notation consistent with classical textbooks (e.g. \cite{kunen}), however in our forcing arguments the stronger conditions are larger (in this aspect, we follow \cite{abrahamiso} as we build on the forcing iteration framework there).

For a set $X$, we let $\Fn(X,n)$ denote the set of finite partial functions from $X$ to $n$. For a partial order $\mc P$ and countable set $N$, we let $\gen{N}{\mc P}$ denote the set of $N$-generic filters in $\mc P$ i.e., filters $H\subs \mc P$ so that any dense $D\subset \mc P$ with $D\in N$ meets $H$. The Rasiowa-Sikorski lemma says that, for any countable $N$, $\gen{N}{\mc P}$ is non empty.
\medskip

Now, let us review the basic notions we will use, and state the most important facts that will help us prove our main result.

\subsection*{Trees} 
By \emph{tree}, we mean a partially ordered set $(T,\leq_T)$ so that $t^\downarrow=\{s\in T:s<_T t\}$ is well ordered by $\leq_T$.\footnote{We usually omit the subscript from $\leq_T$ if it leads to no confusion.} We write $T_\alpha$ for the set of $t\in T$ such that $t^\downarrow$ has order type $\alpha$ (where $\alpha$ is an ordinal); these sets are the levels of $T$, and the height of $T$ is the minimal $\alpha$ so that $T_\alpha=\emptyset$. If $t\in T_\alpha$ and $\xi<\alpha$ then we let $t\uhr \xi$ denote the unique element of $t^{\downarrow}\cap T_\xi$.

An $\aleph_1$-tree is a tree $T$ of height $\omg$ with countable levels. By a \emph{subtree} $S$ of $T$, we will always mean a downward closed subset $S\subs T$ which is also pruned: for any $s\in S_\alpha$ and $\alpha<\beta$ less than the height of $T$, there is $t\in S_\beta$ with $s\leq t$. Given $s\in T$, we will write $T_s$ for the subtree $\{t\in T:s\leq t\}$. 

Given a function $f$ with $\dom f\subs T$ downward closed, we write $\htt(f)$ for the height of $\dom f$ with the tree order inherited from $T$.

\medskip

Now, an \emph{Aronszajn tree} is an $\aleph_1$-tree with no uncountable chains, and a \emph{Suslin tree} is an Aronszajn tree with no uncountable antichains (see \cite{kunen} on the existence of such objects). 
The correspondence between trees and linearly ordered sets are described in detail in \cite{stevotrees}, but we will not really use that analysis in our work.
\medskip

We call a tree $R$ \emph{full Suslin} if all its derived trees $R'=R_{a_1}\otimes \dots \otimes R_{a_n}$ are Suslin where $a_i\neq a_j\in R_\delta$ for $1\leq i<j\leq n$ and a fixed $\delta<\omg$.\footnote{$R_{a_1}\otimes \dots \otimes R_{a_n}=\{(t_i)_{1\leq i\leq n}:a_i\leq_R t_i\in T_\vareps$ for some fixed $\vareps<\omg\}$ with the coordinate wise partial order.}
We write $\partial R$ for the set of all derived trees $R'$ of $R$. The original rigid Suslin tree constructed by Jensen using $\diamondsuit$ is full Suslin actually \cite[Theorem V.1]{devlin2006souslin}.

Given two trees $S,T$, a \emph{club-embedding of $T$ into $S$} is an order preserving injection $f$ defined on $T\uhr C= \bigcup\{T_\alpha:\alpha\in C\}$ where $C\subs \omg$ is a club (closed and unbounded subset), with range in $S$.

We will use two crucial facts on full Suslin trees. First, full Suslin trees are rigid in a strong sense:

\begin{fact}\label{lm:rigid} Suppose that $R$ is a full Suslin tree, $D\subseteq \omg$ is a club, and $s\neq t\in S$ are of the same height. Then there is no order preserving injection from  $R_s\uhr D$ to $R_t\uhr D$. In particular, any order preserving injections $R_s\uhr D\to R_s\uhr D$ must be the identity.
\end{fact}

This result can be extracted from \cite[Lemma 6.7]{stevotrees} or \cite[Theorem V.1]{devlin2006souslin} and the proofs there, but let us present an argument for completeness.
\begin{proof}
 Suppose that $f:R_s\uhr D\to R_t\uhr D$ is order preserving. Take a countable elementary submodel $M\prec H(\Theta)$ so that $f,D,R,s,t\in M$. Note that $\delta=M\cap \omg\in D$. Pick any $s\leq s^*\in R_\delta$ and let $t^*=f(s^*)\in R_t$. Now, since $R_s\otimes R_t$ is Suslin, the branch determined by $(s^*,t^*)$ should be $M$-generic for $R_s\otimes R_t$. By the so called product lemma (see \cite[Lemma I.8]{devlin2006souslin}), $t^*$ should be $M[s^*]$-generic for $R_t$. However, $s^*,f\in M[s^*]$ implies that $t^*=f(s^*)\in M[s^*]$, a contradiction.
\end{proof}


The second fact reflects on the preservation of Aronszajn trees when forcing with a full Suslin tree.

\begin{fact}\label{lm:embed}
\cite[Lemma 3.2]{abrahamiso} For any full Suslin tree $R$  and Aronszajn tree $A$, either
\begin{enumerate}
 \item $\force_{R'}$``$A$ is Aronszajn'' for any $R'\in \partial R$, or
 \item some $R'\in \partial R$ can be club-embedded into $A$.
\end{enumerate}

\end{fact}

By the following observation, we can always suppose that club-embeddings are level preserving:

\begin{obs}
 Suppose that $f:T\uhr C\to S$ is a club-embedding of $T$ into $S$. Then there is some  club $D\subs C$ and order and level preserving $\hat f:T\uhr D\to S\uhr D$ that  satisfies $\hat f(t)\leq f(t)$.
\end{obs}
\begin{proof}
 Given $f:T\uhr C\to S$, we can find a club $ D\subs  C$  so that $\delta\in D$ implies that $$f[T\uhr (\delta\cap C)]\subs S_{<\delta}.$$ Now, note that if $t\in T_\delta$ for some $\delta\in D$ then $f(t)\in T\uhr [\delta,\min D\setm (\delta+1))$, so we can let $\hat f(t)=f(t)\uhr \delta$. Now $\hat f(t)<\hat f(t')$ for any $t<t'\in T\uhr D$ and $\hat f$ is as desired. 
\end{proof}

We also need some lemmas from Moore's framework:

\begin{lemma}\label{lm:clubmin}
\cite[Lemma 2.9]{justinmin} If there is a minimal Aronszajn type then there is an Aronszajn tree $A$ which is \emph{club-minimal} i.e. $A$ can be club-embedded into any subtree $S$ of $A$.
 
\end{lemma}

\subsection*{Uniformization}

A ladder system $\underline \eta$ is a sequence $(\eta_\alpha)_{\alpha\in \lim(\omg)}$ so that $\eta_\alpha$ is a cofinal subset in $\alpha$ of order type $\oo$. An $n$-coloring of $\underline \eta$ is a sequence $\underline h=(h_\alpha)_{\alpha\in \lim(\omg)}$ so that $h_\alpha:\eta_\alpha\to n$. We say that  $\underline h$ is a constant colouring if all the $h_\alpha$ are constant, in which case we can code $\underline h$ by an element of $n^{\lim(\omg)}$.

The main definition is the following:
\begin{dfn}

Given some $\aleph_1$-tree $A$ and $n$-colouring $\underline h$ of a ladder system $\underline \eta$, we say that $f$ is an $A$-uniformization of $h$ if $\dom f=S$ is a subtree of $A$ and for any $\alpha\in \lim(\omg)$, $t\in S_\alpha$ and for almost all $\xi\in \eta_\alpha$, $f(t\uhr \xi)=h_\alpha(\xi)$.
 
\end{dfn}

The main use of this definition is the following:

\begin{lemma}\cite[Lemma 3.3]{justinmin}\label{lm:notclubmin}
If CH holds and for a fixed ladder system $\underline \eta$, any constant  2-colouring of $\underline \eta$ has an $A$-uniformization then $A$ is not club-minimal.
\end{lemma}

Moore \cite{justinmin} showed that  CH is consistent with the statement that for any ladder system $\underline \eta$, any $\oo$-colouring of $\underline \eta$ has an $A$-uniformization for any Aronszajn tree $A$.\footnote{This is rather surprising given the fact that CH implies that for any  $\underline \eta$, there is a constant 2-colouring without an $\omg$-uniformization \cite{devlin1978weak}.} In turn, such a model cannot contain minimal Aronszajn lines.

Let us mention a simple, somewhat technical result for later reference.

\begin{lemma}\label{lm:ctblunif} Suppose that $A$ is a tree of countable height and countable levels, $h:A\to 2$ and  $\underline \eta$ is a ladder system on $\htt(A)$. Then for any $\psi \in \Fn(A,2)$, there is some $f:A\to 2$ extending $\psi$ so that  $f(t\uhr \xi)=^* h(t)$ for any $t\in A_\alpha$, limit $\alpha<\htt (A)$ and almost all $\xi\in \eta_\alpha$.
\end{lemma}
\begin{proof} First, if $A$ is isomorphic to an ordinal, then this result is well known. 

Now, in general, we 'force': let $\mc Q$ bet the poset of functions $g$ where there is some $a\in [A]^{<\oo}$ so that $g:a^\downarrow \to 2$ and $g$ uniformizes $h$ on $\underline \eta$. Extension is  containment.

The aforementioned special case for ordinals implies that the set $D_t=\{g\in \mc Q:t\in \dom g\}$ is dense in $\mc Q$ for any $t\in A$. So, we can take a filter  $G\subseteq \mc Q$ which meets all $D_t$ for $t\in A$ (only countably many dense sets), and so $f=\cup G$  is as desired.
\end{proof}

\medskip

 \subsection*{The main lemma}To summarize the above cited results, we have the following:

\begin{mlemma}\label{lm:base}
 Suppose that CH holds, $\underline \eta$ is a ladder system and $R$ is full Suslin. Suppose further that for any Aronszajn tree $A$, either
 \begin{enumerate}
  \item $\force_{R'}$ ``$A$ is not Aronszajn'' for some $R'\in \partial R$, or 
  \item any constant 2-colouring of  $\underline \eta$ has an $A$-uniformization.
 \end{enumerate} Then there are no minimal uncountable linear orders other than $\pm \omg$.
\end{mlemma}
\begin{proof}
 If there is a minimal  uncountable linear order other than $\pm \omg$, then it has to be Aronszajn by the CH, and so there is an Aronszajn tree $A$ which is club-minimal by Lemma \ref{lm:clubmin}. For this particular tree, condition (2) must fail by Lemma \ref{lm:notclubmin}. So $\force_{R'}A$ is not Aronszajn for some $R'\in \partial R$, which implies that there is a club $D_0$, $R'\in \partial R$ and an order preserving embedding $R'\uhr D_0\to A$ (by Fact \ref{lm:embed}).
 
 In particular, there is $s\in R$ with a level preserving club embedding $f_s:R_s\uhr D_0\to A$. Let $S$ denote the downward closure of $f_s[R_s]$ and note that $S\uhr D_0=f_s[R_s]\uhr D_0$. Since $A$ is club-minimal, we can find $D_1\subs D_0$ with a level preserving embedding $g:A\uhr D_1\to S\uhr D_1$. Now, $g$  must fix each point of $f_s[R_s]\uhr D_1$ by Fact \ref{lm:rigid}, and so $A\uhr D_1$ must be equal $f_s[R_s]\uhr D_1$, which is of course isomorphic to $R_s\uhr D_1$. This is a contradiction, since $A\uhr D_1$ must have non-trivial club embeddings ($A$ being club-minimal), while $R_s\uhr D_1$ has no such embeddings  by Fact \ref{lm:rigid}.
 
\end{proof}

Our goal for the rest of the paper is to show that the assumptions of this lemma are consistent with ZFC (assuming that ZFC itself is consistent).

%
%
%
%

\section{The outline of the forcing construction}\label{sec:forceoutline}

Our aim is to construct a sequence of partial orders $\langle T^\tau:\tau<\oo_2\rangle$, that will serve as our forcing iteration, so that, for all $\tau<\oo_2$,
\begin{enumerate}
 \item  $T^\tau$ is a Suslin tree (which will ensure the ccc and that no new reals are added), and
 \item $T^\tau$ forces an $A$-uniformization for some colouring and Aronszajn tree $A$.
\end{enumerate}
In order to make this sequence $\langle T^\tau:\tau<\oo_2\rangle$ an actual iteration, we will ensure that 
\begin{enumerate}
\setcounter{enumi}{2}
\item $T^\tau$ is a \emph{refinement} of $T^\nu$ for $\nu<\tau<\oo_2$ i.e., there is a club $C\subs \omg$ and a so-called projection $\pi:T^\tau\to T^\nu\uhr C$, that is,
\begin{enumerate}[(i)]
                                                                                                                                                                                                                                              \item $\pi$ is an order preserving  surjection, and
                                                                                                                                                                                                                                              \item if $t\in T^\nu\uhr C$ and $t>\pi(s)$ for some $s\in T^\tau$ then $\pi(s')=t$ for some $s<s'\in T^\tau$.
                                                                                                                                                                                                                                             \end{enumerate}
                                                                                                                                                                                                                                             
       \end{enumerate}
      
       Finally, we will also have a full Suslin tree $R$, so that, for any $\tau<\oo_2$,
       
         \begin{enumerate}
\setcounter{enumi}{3}                                                                                                                                                                                                                                    
                                                                                                                                                                                                                                             
                                                                                                                                                                                                                                             \item  $\force_{T^\tau}R$ is full Suslin.
\end{enumerate}


All of this boils down to two separate goals: we need to specify how to construct $T^{\tau+1}$ from $T^\tau$ (that is, the successor stages of the iteration), and how to build limits for such sequences (while preserving $R$ full Suslin and the refinement property above).

\medskip

First, our main theorem handles the successor stages. We let $\mc C$ denote the \emph{club forcing} $$\mc C=\{(\nu,A):\nu<\omg,A\subs \omg \textmd{ is a club}\}$$ ordered by $(\nu,A)\leq (\nu',A')$ if $\nu\leq \nu', A'\subseteq A$ and $\nu\cap A=\nu\cap A'$; from now on, we reserve the letter $C$ for a $V$-generic club with canonical name $\dot C$. $\mc C$ is a countably closed forcing with the property that the generic club $C$ is contained mod countable in any ground model club $D$. 

Recall that  $\diamondsuit^*$ asserts the existence of a sequence $W=(W_\delta)_{\delta<\omg}$ so that $|W_\delta|\leq \oo$ and for any $X\subs \omg$, the set $\{\delta<\omg:X\cap \delta\in W_\delta\}$ contains a club.

\begin{mthm}\label{thm:main} Suppose that $V$ is a model of $\diamondsuit^*$, and
\begin{enumerate}
                          \item $T$ is a Suslin tree, and $\underline \eta$ is a ladder system in $V$,
                          \item  $V^T\models$``$R$ is full Suslin'', 
                          \item $\dot A, \dot h$ are $T$-names so that $V^{T}\models \dot h\in 2^\omg$, and, for any $R'\in \partial R$, $$V^{T\times R'}\models \dot A \textmd{ is an Aronszajn-tree}.$$ 
                         \end{enumerate}

                         Then, in $V[C]$, there is a refinement $\tilde T$ of $T$ so that 
                         \begin{enumerate}[(a)]
                          \item $V[C]^{\tilde T}\models R$ is full Suslin, and
                          \item $V[C]^{\tilde T}\models$ the constant 2-colouring coded by $\dot h$ on $\underline \eta$ has an $\dot A$-uniformization.
                         \end{enumerate}
\end{mthm}

The proof will be presented in the next section, but let us show the reader how this theorem can be applied to prove our main result on linear orders.

\medskip

The Main Theorem above is complemented by Abraham and Shelah's iteration theorem that we include here for ease of reference;\footnote{This iteration theorem in \cite{abrahamiso} was used to show that consistently, CH holds and there is a full Suslin tree $R$ and special Aronszajn tree $U$ so that, for any Aronszajn tree $A$, either $A$ embeds into $U$ on a club, or there is a derived tree of $R$ that club-embeds into $A$.} let us say that $\sigma$ is an \emph{s-operator}\footnote{Short for successor-operator.} if $\sigma$ is defined on $<\oo_2$ sequences of Suslin trees $\mc T=\langle T^\nu:\nu\leq \tau\rangle$ so that $\sigma(\mc T)$ is a refinement of $T^\tau$. We say that $\sigma$ is $R$-preserving if  $\force_{T^\tau}$ ``$R$ is full Suslin'' implies that $\force_{\sigma(\mc T)}$ ``$R$ is full Suslin''.
Note that our Main Theorem above stated the existence of a particular $R$-preserving s-operator.

\begin{thm}\cite[Theorem 4.14]{abrahamiso}\label{thm:ASh} Suppose that $\diamondsuit$ and $\square$ holds,\footnote{Let us skip the definition of $\diamondsuit$ (a weaker form of $\diamondsuit^*$), and $\square$ and refer the reader to \cite{abrahamiso,devlin2006souslin} since we will only use them to apply this theorem.} and $R$ is a full Suslin tree. Given any $R$-preserving s-operator $\sigma$, there is a sequence  $\langle T^\tau:\tau<\oo_2\rangle$ so that, for any $\tau<\oo_2$,
\begin{enumerate}[(i)]
 \item $\force_{T^\tau}R$ is full Suslin,
 \item $T^{\tau+1}=\sigma(\langle T^\nu:\nu\leq \tau\rangle)$, and
 \item $T^\tau$ is a refinement of $T^\nu$ for all $\nu<\tau$.
\end{enumerate}

\end{thm}

\medskip

Finally, putting together the two latter theorems yields the following corollary.

\begin{mcor} \label{cor:main}
 Consistently, CH holds and there is a full Suslin tree, while there are no minimal uncountable linear orders other than $\pm \omg$.
\end{mcor}
\begin{proof} Our goal is to find a model where the assumptions of Lemma \ref{lm:base} are satisfied.
 We start from the constructible universe $L$ (where $\diamondsuit$ and $\square$ holds), and first iterate $\mc C$ with countable support in length $\oo_2$ to add generic clubs $(C_\alpha)_{\alpha<\oo_2}$. We denote the resulting model with $V$, which still satisfies $\diamondsuit$ and $\square$ (see \cite{abrahamiso} for why we can do this). Moreover, any intermediate model $V_\alpha=L[(C_\nu)_{\nu<\alpha}]$ will satisfy $\diamondsuit^*$ as well (for $\alpha<\oo_2$) \cite[Chapter X, Lemma 1]{devlin2006souslin}.
 
 Let $R$ be a full Suslin tree in $V$, and $\underline \eta$ an arbitrary ladder system on $\omg$. Now, we construct an $R$-preserving s-operator in $V$ using our Main Theorem, which in turn gives an iteration sequence of Suslin trees $\langle T^\tau:\tau<\oo_2\rangle$  by the Abraham-Shelah theorem.
 
 To define $\sigma$, suppose we are given some $\mc T=\langle T^\nu:\nu\leq \tau\rangle$ so that $\force_{T^\tau}$ ``$R$ is full Suslin''. An appropriate bookkeeping hands us a $T^\tau$-name $\dot A$ for an Aronszajn tree and $\dot h$ coding a constant 2-colouring of $\underline \eta$. Find the minimal $\alpha<\oo_2$ so that $T^\tau,\dot h, \dot A\in V_\alpha$. 
 
 Now, if $V_\alpha^{T^\tau\times R'}\models$ ``$\dot A$ is Aronszajn'' for all $R'\in \partial R$ then we apply our Main Theorem in $V_\alpha$ to find a refinement $\sigma(\mc T)$ of $T^\tau$ in  $V_\alpha[ C_\alpha]$, so that $V_\alpha[C_\alpha]^{\sigma(\mc T)}\models$ ``$R$ is full Suslin, and the colouring coded by $\dot h$ on $\underline \eta$ has an $\dot A$-uniformization''. Otherwise, we just let $\sigma(\mc T)=T^\tau$.

Now $\langle T^\tau:\tau<\oo_2\rangle$ is given to us by Theorem \ref{thm:ASh}, and we let $T$ be the direct limit of this sequence. We claim that the model $V^T$ is as desired. Indeed, note that $T$ is ccc of size $\aleph_2$, and  $T$ adds no new reals (since any new set of size at most $\aleph_1$ must be introduced by $T^\tau$ for some $\tau<\oo_2$, and $T^\tau$ is a Suslin tree so adds no new reals). Furthermore, for any tree $A$ in $V^T$, either $A$ embeds some derived tree of $R$ on a club or $A$ remains Aronszajn after forcing with any $R'\in \partial R$. In the latter case, for any colouring $h$ of $\underline \eta$, there was an intermediate stage when we uniformized $h$ on $A$.
\end{proof}

This leaves us to prove the Main Theorem. In general, for notions that we might have left undefined in the present paper, or for additional background, let us refer the reader to the very well written papers of Abraham and Shelah \cite{abrahamiso} and to Keith Devlin and Havard Johnsbraten's book on the Suslin problem \cite{devlin2006souslin}, where Jensen's original model of CH with no Suslin-trees is detailed.

\section{The proof of the main theorem}\label{sec:proof}

 
 The current section is devoted entirely to show our Main Theorem, which we break down into a few reasonable segments.
 
 Let us recall the setting first: we have a model $V$, a Suslin tree $T$ considered as a forcing notion, and names $\dot A$ for an Aronszajn tree and $\dot h$ for an element of $2^\omg$.

\subsection*{Some preparations}

First, working in $V$, find a club $F\subs \omg$ and $A(x),h(x)$ for $x\in T_\gamma$ with $\gamma\in F$ so that 
\begin{enumerate}
 \item $x\force_T$ ``$ \dot A\uhr \gamma =A(x)$ and $\dot h\uhr \gamma=h(x)\in 2^\gamma$'',
 
 \item there is a maximal antichain $\mc T_x\subseteq T$ above $x$, countable sets $B_z$  and $i_z\in 2$ for $z\in \mc T_x$ so that, for any $z\in \mc T_x$, 
 \begin{enumerate}
  \item $\htt(z)<\min F\setm (\gamma+1)$ and $z$ decides $\dot A_\gamma$,
  \item $b\in B_z$ iff $z\force_T $ ``$b$ is a branch in $A(x)$ which has an upper bound in $\dot A_\gamma$'', and
  \item $z\force_T$ $ \dot h(\gamma)=i_z$.
 \end{enumerate}\end{enumerate}

 In other words, $(B_z)_{z\in \mc T_x}$ collects the countably many possibilities that can be forced (above $x$) for the $\gamma$th level of $\dot A$. We set $B(x)=\cup\{B_z:z\in \mc T_x\}$. 
 Note that for any $z\in \mc T_x$, and any node $t\in A(x)$, some $b\in B_z$ extends $t$ (which can be written concisely as $\bigcup B_z=A(x)$).
 
 \medskip

  Now, in $V$, take a $\diamondsuit^*$ sequence $W=(W_\delta)_{\delta<\omg}$ which remains a $\diamondsuit^*$ sequence after forcing with $T\times R'$ for any $R'\in \partial R$, and let $W^*_\delta\supseteq W_\delta$ in $V[C]$ (where $C$ is the generic club added by the forcing $\mc C$) so that $W^*$ is still a $\diamondsuit^*$ sequence after forcing with $R'$ for any $R'\in \partial R$ (let us refer the reader to \cite[Chapter VIII]{abrahamiso} for details on why this is possible).
 
 \medskip
 
 Recall that $\mc C$ was the club forcing and $C$ is the $V$-generic club. Working in $V[C]$, the generic club $C$ is mod countable contained in $F$, so we  let $\gamma_0=0$ and let $\{\gamma_\alpha:1\leq\alpha<\omg\}$ enumerate an end-segment of $C$ that is contained in $F$. The fact that  $\tilde T$ refines $T$ will be witnessed by a projection $\tilde T\to T\uhr \{\gamma_\alpha:\alpha<\omg\}$.\footnote{To remind the reader, we need a refinement because our main theorem provides the successor steps of an iteration.}

\medskip

\subsection*{How will the elements of $\tilde T$ look like?} The $\alpha$th level $\tilde T_\alpha $  of the tree $\tilde T$ will consist of  pairs $(x,f)$ so that $x\in T_{\gamma_\alpha}$ and $f:S\to 2$ so that $S\subseteq A(x)$ is downward closed and pruned, and $f$ is a uniformization of the coloring coded by $h(x)\in 2^{\gamma_\alpha}$ on the ladder system $\underline \eta\uhr\gamma_\alpha$.\footnote{I.e., for any $t\in S_\delta$, $f(t\uhr \xi)=h(x)(\delta)$ for almost all $\xi\in \eta_\delta$.} We will ensure that $(x,f)\mapsto x$ is the projection that witnesses that $\tilde T$ is a refinement of $T$.

The extension in $\tilde T$ is defined as follows: $(x,f)\leq (x',f')$ if $x\leq x'$ in $T$ and $f\subseteq f'$. We would like the second coordinates to introduce the desired uniformization, and hence we need that for any $(x,f)$ and $\delta<\omg$, there is some $ (x',f')$ above $(x,f)$ so that $\dom f'$ has height at least $\delta$.

In turn, we require the following \emph{richness property} (RP): 
\medskip
 \begin{center}
  \begin{minipage}[c]{0.7\textwidth}
for any $(x,f)\in \tilde T$ and $z\in\mc T_x$, each  $t\in \dom f$ is extended by some branch  $b\in B_z$ so that $b\subseteq \dom f$ and $f(b\uhr \xi) =i_z$ for almost all $ \xi\in \eta_{\gamma_\alpha}.$ 
 \end{minipage}
 
 \end{center}

\begin{clm}\label{clm:unif}
 Any $\tilde T$ of the above described form with the (RP) will introduce an $\dot A$-uniformization for the constant 2-colouring coded by $\dot h$.
\end{clm}

\begin{proof}  
 A $V[C]$-generic filter $G\subs \tilde T$ defines a generic branch $x\subs T$ that evaluates $\dot A$ to be $\dot A[x]=\cup\{A(x\uhr \gamma_\alpha):\alpha<\omg)\}$, and $\dot h$ is evaluated as $\dot h[x]=\cup\{h(x\uhr \gamma_\alpha):\alpha<\omg\}$. The union $f$ of the second coordinates in $G$ defines a function on a subset of $\dot A[x]$ that uniformizes $\dot h[x]$. Finally, we need that $\dom f$ is really a subtree in our sense, which comes down to showing that any condition in $\tilde T$ has arbitrary high extensions, so by genericity, $\dom f$ must be pruned.

That is, we would like to show that the set of conditions $(x',f')$ such that $\dom f'$ has height at least $\delta$ is dense for any $\delta<\omg$.
 Given $(x,f)$ and $\delta<\omg$, we first find $x<x'\in T_{\gamma_\beta}$ with $\gamma_\beta>\delta$. Then, we can take the unique $z\in \mc T_x$ which is compatible with $x'$; note that  $$x'\force_T \dot h\uhr \gamma_\alpha+1=h(x)\cup\{(\gamma_\alpha, i_z)\}.$$
Now, for each $t\in \dom f$, pick $b_t\in \{b\in  B_z:f(b\uhr \xi) =i_z \textmd{ for almost all }\xi\in \eta_{\gamma_\alpha}\}$ so that $t\in b_t$. Extend $\dom f$ by adding unique upper bounds for all the branches $b_t$ from $A(x')_{\gamma_\alpha}$, and let this set be $S_0\supseteq \dom f$. Note that any function $f_0:S_0\to 2$ that extends $f$ is still a uniformization for $h(x)\cup\{(\gamma_\alpha, i_z)\}$. So, pick some pruned $S'\subs A(x')$ that extends $S_0$, and using Lemma \ref{lm:ctblunif},  find $f'\supseteq f_0$ with domain $S'$ such that  $(x,f)\leq (x',f')\in \tilde T$.

\end{proof}

 \subsection*{The extension property} Finally, we need to ensure that the map $(x,f)\mapsto x$ from $\tilde T$ onto $T\uhr \{\gamma_\alpha:\alpha<\omg\}$ is a projection. So, our goal will be to ensure that if $(x,f)\in \tilde T_\alpha$ and $x<x'\in T_{\gamma_\beta}$ for $\alpha<\beta$, then there is some $f'$ so that $(x,f)\leq (x',f')\in \tilde T_\beta$. In fact, we need a stronger property to carry out the inductive construction of $\tilde T$.

 Suppose $x<x'\in T$ and $b\in B(x')$ is a cofinal branch through $A(x')$. If $f:S\to 2$ for some pruned, downward closed $S\subseteq A(x)$ then we say that $b$ is \emph{$x'$-compatible} with $f$ if $$b\uhr \gamma_{\alpha}\subseteq S \textmd{ and } f(b\uhr \xi)=i \textmd{ for almost all  }\xi\in \eta_{\gamma_\alpha} \textmd{ and } i=h(x')(\gamma_\alpha).$$  
 This simply means  that if $x$ forced $f$ to be a uniformization of $\dot h$ so far, then we have the possibility to add an upper bound of $b\uhr \gamma_\alpha$ to $\dom f$ without running into trouble at level $\gamma_\alpha$  with the uniformization (in the universe forced by $x'$). We will say that a finite function $p\in \Fn(B(x'),2)$ is $x'$-compatible with $f$ if each  $b\in\dom p$ is $x'$-compatible with $f$.

 Now, we will assume inductively and preserve the following \emph{extension property} (EP) along the construction of $\tilde T$:
   \medskip
 
   \begin{center}
  \begin{minipage}[c]{0.7\textwidth}
   for any $\alpha<\beta<\omg$, $(x,f)\in \tilde T_\alpha$ and $x<x'\in T_{\gamma_\beta}$, if $p\in \Fn(B(x'),2)$ is $x'$-compatible with $f$ and $\psi\in \Fn(A(x')\setm A(x),2)$  then there is $f'$ so that $(x,f)\leq (x',f')\in \tilde T_\beta$ and 
 \begin{enumerate}
  \item for all $b\in \dom p$, $b\subseteq \dom f'$ and $f'(b\uhr \xi)=p(i)$ for almost all $\xi\in  \eta_{\gamma_\beta}$, and
  \item $\psi(t)=f'(t)$ whenever both are defined.
  \end{enumerate}
  \end{minipage}
  
   \end{center}
  
   \begin{figure}[H]
  \includegraphics[width=.5\textwidth]{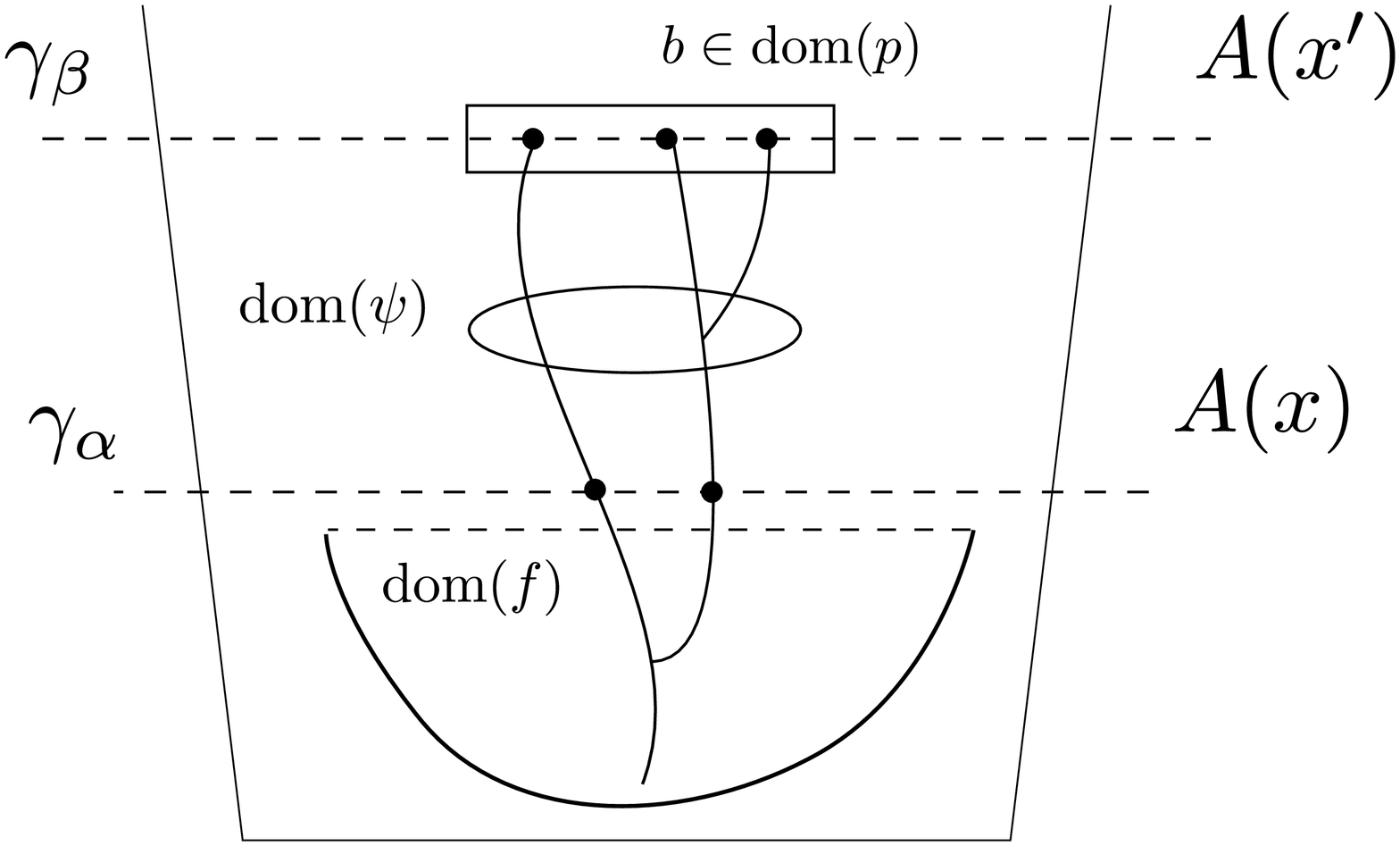}
    \caption{The (EP)}
  \label{diag:ep}
  

 \end{figure}

 Note that only those points $t\in \dom \psi$ matter where the branch $t^\downarrow$ is compatible with $f$, and by extending $p$ we can make sure $f'$ is defined on these $t$ as well. In other words, we can suppose that $\dom \psi\subs \cup \dom p$ (see Figure \ref{diag:ep} for the setting).

 \medskip
 
 \subsection*{The construction of $\tilde T$} The levels $\tilde T_\alpha $ of $\tilde T$ will be constructed by induction on $\alpha<\omg$ preserving the (RP) and (EP). In fact, $\tilde T_\alpha $ will be a result of taking $N_\alpha$-generic filters for appropriate posets (approximating the conditions in $\tilde T_\alpha$), where $(N_\alpha)_{\alpha<\omg}$ is a canonically chosen sequence of  countable elementary submodels. In fact, we let  $N_\alpha$ be $$L_\delta[T\uhr {\gamma_\alpha+1},A\uhr {\gamma_\alpha+1}, R\uhr \alpha+1, C\cap \gamma_\alpha, W^*\uhr \gamma_\alpha+1]$$ for the unique minimal $\delta>\alpha$ which makes this a model of $ZF^-$. This ensures that any model $N$ with the parameters $T\uhr {\gamma_\alpha+1},A\uhr {\gamma_\alpha+1},\dots$ actually contains $N_\alpha$ as an element (and subset). 

 \subsection*{Successor steps -  $\tilde T_{\alpha+1}$ from $\tilde T_\alpha$} Given $\tilde T_\alpha$, we will construct $\tilde T_{\alpha+1}$ while preserving the (RP) and (EP). Fix some $(x,f)\in \tilde T_\alpha$ and $x<x'\in T_{\gamma_{\alpha+1}}$. Define the poset $$\mc P_{x',f}=\{p\in \Fn(B(x'),2):p \textmd{ is } x'\textmd{-compatible with }f\}$$ where extension is simply containment. For each $p_0\in \mc P_{x',f}$, we take a minimal\footnote{Minimal with respect to a fixed well-order of $L_{\omg}$.} $H\in \gen{N_{\alpha+1}}{\mc P_{x',f}}\cap N_{\alpha+2}$ so that $p_0\in H$. Let $p^H=\cup H$, and let $S=\dom f$.
 
\begin{clm} The set $S'=\cup \dom p^H$ is a pruned and  downward closed subtree of $A(x')$, and $S=S'\cap A(x)$.
 \end{clm}
 
 \begin{proof}
Clearly, $S'$ is pruned and downward closed as a union of branches through $A(x')$.
 
 The fact that any $p\in H$ is compatible with $f$ implies that $b'\uhr \gamma_\alpha\subs S$ for $b'\in \dom p$ and so $S\subs S'\cap A(x)$. Note that if $z\in \mc T_x$ is compatible with $x'$ then for any $s\in S$, there are infinitely many branches $b\in B_z$ through $A(x)$ so that $s\in b$ and $f(b\uhr\xi)=i_z$ for almost all $\xi\in \eta_{\gamma_\alpha}$; let $B_s^*$ denote these branches. Any $b\in B_s^*$ has an upper bound in $A(x')_{\gamma_\alpha}$ which is extended to branches $b'\in B(x')$. So, some $b\in B_s^*$ has an extension in $S'$ by genericity and so  $S=S'\cap A(x)$.


 \end{proof}

 Now, we use $p^H$ to form an $f'\supseteq f$ defined on $S'$ that satisfies the (RP).
 
 \begin{clm} For any $\psi$ compatible with $f$, there is an $f':S'\to 2$  so that 
 \begin{enumerate}
 \item $f\subseteq f'\in N_{\alpha+2}$ and $f'$ uniformizes $h(x')$ on $\underline \eta\uhr \gamma_{\alpha+1}$,
 \item $\psi \subs f'$, and
 \item\label{item:rp} $f'(b\uhr\xi) =p^H(b)$ for all $b\in \dom p^H$, and almost all $\xi \in \eta_{\gamma_{\alpha+1}}$. 
 \end{enumerate}
 \end{clm} 
 \begin{proof}
This is simply done by Lemma \ref{lm:ctblunif}. 
  \end{proof}
  
 \begin{clm}
 $f'$ satisfies the (RP).
 \end{clm}
  \begin{proof}
   We again use the genericity of $H$: given $z'\in \mc T_{x'}$, we need to show that for any $s\in S'$ there is some $b\in B_{z'}$ extending $s$ so that $f'(b\uhr\xi)=i_{z'}$ for almost all $\xi\in \eta_{\gamma_{\alpha+1}}$. It suffices to show, by condition (\ref{item:rp}) of $f'$, that there is some $b\in \dom p^H\cap B_{z'}$ extending $s$ so that $p^H(b)=i_{z'}$. By genericity of $H$ (and since $A(x'),\mc T_{x'}\in N_{\alpha+1}$), it suffice that there are infinitely many $b\in B_{z'}$ that extend $s$; however, this clearly holds since $\force_T\dot A$ is pruned and $B_{z'}$ just collects the branches in $A(x')$ that are forced to be bounded in $\dot A$.
  \end{proof}

 
 Now, we put $(x',f')\in \tilde T_{\alpha+1}$. The function $f'$ depended on the initial choice of $p_0\in \mc P_{x',f}$ and on $\psi$, and we do this for all countably many possible choices. This  in turn defines $\tilde T_{\alpha+1}$ (in $N_{\alpha+2}$) in a way that the (EP) is preserved.
 
 \subsection*{Limit steps - $\tilde T_{\beta}$ from $\tilde T_{<\beta}$} Suppose that $\tilde T_{<\beta}=\bigcup\{\tilde T_\alpha:\alpha<\beta\}$ is already constructed, and fix some $x'\in T_{\gamma_\beta}$. We will now force with the poset $\mc P_{x'}$ of all pairs $(p,f)$ so that 
 \begin{enumerate}
  \item $(x,f)\in \tilde T_{<\beta}$ for some $x<x'$, 
  \item $p\in \Fn(B(x'),2)$ is $x'$-compatible with $f$, and 
    \item $b\uhr \htt(f)\neq b'\uhr \htt(f)$ for any $b\neq b'\in \dom p$.
 \end{enumerate}

  Extension is defined by $(p,f)\leq (\bar p,\bar f)$ if $p\subseteq \bar p$, $f\subseteq \bar f$ and for any $b\in \dom p$, 
 \begin{enumerate}
 \setcounter{enumi}{3}
  \item\label{item:ext} $\bar f(b\uhr \xi)=p(b)$ for any $\xi\in \eta_{\gamma_\beta}\cap \htt \bar f\setm \htt f$.
 \end{enumerate}

 
 Given some $(p_0,f_0)\in \mc P_{x'}$, we take a minimal  $H\in \gen{N_{\beta}}{\mc P_{x'}}\cap N_{\beta+1}$ with $(p_0,f_0)\in H$. Let $$f'=\bigcup\{f:(p,f)\in H\}$$ and $p^H=\bigcup\{p:(p,f)\in H\}$. Let $S'=\dom f'$.
 
 \begin{clm}
  $S'=\cup \dom p^H$ is a pruned, downward closed subtree of $A(x')$. Furthermore, for any $b\in \dom p^H$, $f'(b\uhr \xi)=p^H(i)$ for almost all  $\xi\in \eta_{\gamma_\beta}$.
 \end{clm}
 \begin{proof} First, $\dom f'\supseteq \cup \dom p^H$ holds since $b\uhr \htt(f)\subs \dom f$ for any $b\in \dom p$ and $(p,f)\in H$. To see the reverse inclusion, just note that for any $(p,f)\in \mc P_{x'}$ and $s\in \dom f$, there are infinitely many $b\in B(x')$ that are $x'$-compatible with $f$ and extend $s$. So, by genericity, we included some of these in $\dom p^H$.

  The latter statement is clear from the way we extend conditions in  $\mc P_{x'}$ (see condition (\ref{item:ext}) above).
 \end{proof}

  \begin{clm}
   $f'$ satisfies the (RP).
  \end{clm}
 \begin{proof} 
  Let $z'\in \mc T_{x'}$, and we need to show that for any $s\in S'$ there is some $b\in B_{z'}$ extending $s$ so that $f(b\uhr\xi) =i_{z'}$ for almost all $\xi\in \eta_{\gamma_{\beta}}$. It suffices that there is some $b\in \dom p^H\cap B_{z'}$ extending $s$ so that $p^H(b)=i_{z'}$. By genericity of $H$ (and since $A(x'),\mc T_{x'}\in N_{\beta}$), we need that there are infinitely many $b\in B_{z'}$ that extend $s$; however, this clearly holds since $\force_T\dot A$ is pruned.
 \end{proof}

Now, we put $(x',f')\in \tilde T_\beta$, and we repeat this for all possible choices of $(p_0,f_0)\in \mc P_{x'}$ (again, we only have countably many such), which in turn defines $\tilde T_\beta$. 
 
  \begin{clm}
   The (EP) is preserved.
  \end{clm}
  \begin{proof} 
  Indeed, given some $(x,f)$, $p_0$ and $\psi$ we can first use the (EP) for $\tilde T_{<\beta}$ to find $(x,f)\leq (x_0,f_0)$ that is still compatible with $p_0$ and $\psi\subseteq f_0$. Now, take $f'$ that corresponds to the filter $H$ that we chose for $(p_0,f_0)$. Then $(x',f')$ witnesses the (EP).
   \end{proof}
  \medskip
  
  This finishes the construction of $\tilde T=\cup_{\alpha<\omg}\tilde T_\alpha$, which is  an $\aleph_1$-tree (with the coordinate wise ordering) and  certainly a refinement of $T$ by the (EP). Also, we proved already in Claim \ref{clm:unif} that, in $V[C]^{\tilde T}$, the colouring coded by $\dot h$ on $\underline \eta$ has an $\dot A$-uniformization.\footnote{At this point, using the $\diamondsuit^*$ sequence $W^*$ that was in the models $N_\alpha$, it would be standard to show that $\tilde T$ is Suslin (see \cite[Chapter IV, Lemma 2]{devlin2006souslin}). Hence, together with Jensen's iteration framework, we arrive to an alternative proof to Moore's result: the consistency of CH with no minimal uncountable linear orders other than $\pm \omg$.}
  

 \subsection*{Why is $R$ still full Suslin after forcing with $\tilde T$?} This will be the crux of the proof, where we simultaneously show that $\tilde T$ is Suslin, and that $\tilde T$ preserves $R$ full Suslin. 
 In order to do this, it suffices to prove that, in $V[C]$, $\tilde T\times R'$ is Suslin for any derived tree $R'$ of $R$. 
 
 If this were not the case, and we let $\dot R'$ denote a $V[C]$-generic branch for $R'$, then   $$V[C][\dot R']\models \tilde T\textmd{ is not Suslin}$$ for some $R'\in \partial R$. Let $\dot X$ be a name for a maximal antichain of $\tilde T$; as usual, we would like to find an $\alpha<\omg$ so that $\dot X\uhr \alpha=\dot X\cap T_{<\alpha}$ is maximal already in $\tilde T$ (and so $\dot X=\dot X\uhr \alpha$ is countable). In turn, we are in search for an $\alpha<\omg$ so  that any $(x',f')\in \tilde T_\alpha$ extends some element of $\dot X\uhr \alpha$.
 
 We need a few simple claims to prepare our argument.
 
 \begin{clm}
$V[C][\dot R']=V[\dot R'][C]$, and the club $C$ is also $V[\dot R']$-generic.
  \end{clm}
  \begin{proof}
  
 Indeed, on one hand $R'$ is ccc so any club of $\omg$ in $V^{R'}$ contains  a club from $V$. Also, $R'$ introduces no new $\oo$-sequences, so the poset $\mc C$ is the same in $V$ and $V^{R'}$. 
\end{proof}
 
 Now, working in $V[\dot R']$, we take  countable elementary submodels $N\in M\prec H^{V[\dot R']}_{\oo_3}$ with $T,R,A,W^*,\mc C,\dot X\dots \in N$. Let $\pi_N,\pi_M$ denote the collapsing functions for $N$ and $M$, and let let $\bar N$ and $\bar M$ denote the transitive collapses of $N$ and $M$, respectively.  We will prove that  $\alpha=N\cap \omg$ satisfies our requirements.
 
 Let us cite two results \cite[Chapter IX, Lemma 2 and 3]{devlin2006souslin}:

 \begin{clm}
  $C\cap \alpha$ is $\bar M$-generic for the poset $\pi_N(\mc C)=\{(\nu,B\cap \alpha):(\nu, B)\in \mc C\cap N\}$.
 \end{clm}

  \begin{clm} For any formula $\varphi$ with constants from $\{\check x:x\in N\}\cup \{C\}$,
    $$H_{\oo_3}^{V[\dot R'][ C]}\models \varphi  \textmd{ if and only if }   \bar N[C\cap \alpha]\models \pi_N(\varphi).$$
  
In turn, $\pi^{-1}_N$ extends to an elementary embedding $$\pi^{-1}_N:\bar N[C\cap \alpha]\to H_{\oo_3}^{V[\dot R'][ C]}$$ that maps $ C\cap \alpha$ to $C$.
  \end{clm}
 
Recall that in the construction of $\tilde T$, we worked with the canonical model sequence $(N_\beta)_{\beta<\omg}$, defining the levels of $\tilde T$ by choosing minimal generic filters for certain posets. 
  Since $\bar N$ contains the relevant parameters, we can carry out the same construction in  $\bar N[C\cap \alpha]$, and  hence, $\bar N[C\cap \alpha]$ contains  the tree $\tilde T\uhr \alpha$.  
  

 \begin{clm}\label{clm:guess}
 There is a club $D\subs \omg$ so that $D\in \ran \pi_N^{-1}$ and $\beta\in D$ implies $\dot X\uhr \beta\in N_\beta$. In turn, the set $\{\beta\leq \alpha:\dot X\uhr \beta \in N_\beta\}$ is closed and unbounded in $\alpha+1$ and so $\dot X\uhr \alpha\in N_\alpha$.
 
   \end{clm}
\begin{proof} 

   Recall that $W^*$ was a $\diamondsuit^*$ sequence in $V[\dot R'][ C]$, so there is a club $D$ in $\omg$ so that $\beta \in D$ implies $\dot X\uhr \beta\in W^*_\beta\subseteq N_\beta$. By elementarity, there is $E\in \bar N[C\cap \alpha]$ so that $D=\pi^{-1}_N(E)$ satisfies the above requirements. Since $D\cap \alpha=E$ and $D$ was closed unbounded, the claim follows.


 
\end{proof}

By elementarity, $\dot X\uhr \alpha$ is a maximal antichain in $\tilde T\uhr \alpha$.

\medskip

We will also need the next claim, where, given $x'\in T_\alpha$, we let $\dot x'$ denote the branch $\{x\in T:x<x'\}$ of $T\uhr \alpha$.

   \begin{clm}
     For any $x'\in T_\alpha$, $\bar N[ C\cap \alpha][\dot x']\models$ ``$\dot A\uhr \alpha$ is Aronszajn''.
   \end{clm}
   
\begin{proof}
 Since $V[R']\models$ ``$T$ is Suslin'', we also have $\bar N\models$``$T\uhr \alpha$ is Suslin''. This is preserved by $\sigma$-closed forcing, so $\bar N[ C\cap \alpha]\models$ ``$ T\uhr \alpha$ is Suslin''. Now, $\dot x'$ is an $\bar N[C\cap \alpha]$-generic branch for the tree $T\uhr \alpha$. As $\dot A$ was  a $T$-name for a tree that is Aronszajn in $V^{R'\times T}$, we also have  $\bar N[ C\cap \alpha][\dot x']\models$ ``$ \dot A\uhr \alpha$ is Aronszajn'' by elementarity.
\end{proof}

After all this preparation, lets show that whenever $(x',f')\in \tilde T_\alpha$ then $(x',f')$ is above some element of $\dot X\uhr \alpha$. Recall that $f'$ was constructed using an $N_\alpha$-generic filter for the poset $\mc P_{x'}$. In turn, we will aim for a density argument, and it suffices to show that the following claim holds.
   
   \begin{mclm}\label{lm:density}Let  $\mc D$ be the set of all $(p,f)\in \mc P_{x'}$ so that $(x,f)\in \tilde T\uhr \alpha$ for some $x<x'$ and $(x,f)$ is above an element of $\dot X\uhr \alpha$. Then $\mc D\in N_\alpha$ and $\mc D$ is dense in $ \mc P_{x'}$.

   \end{mclm}

\begin{proof} First, note that  $\mc D\in N_\alpha$ follows from $\dot X\uhr \alpha \in N_\alpha$.
 
  Now, suppose that $\mc D$ is not dense, and we reach a contradiction. That is, we assume that  some $(p_0,f_0)\in \mc P_{x'}$ has no extension in $\mc D$. There is some $x_0<x'$ so that $(x_0,f_0)\in \tilde T\uhr \alpha$, and let $\gamma_\tau<\alpha$ so that $x_0\in T_{\gamma_\tau}$. 
 \medskip

 In this proof, we will say that $q$ is \emph{bad}, if $q\subs A(x')_\beta$ for some $\gamma_\tau\leq \beta<\alpha$ and

 \begin{enumerate}
  \item $|q|=|p_0|$,
  \item $(\cup \dom q) \uhr \gamma_\tau=(\cup \dom p_0)\uhr \gamma_\tau$ (which implies that $q$ is $x'$-compatible with $f_0$), and
  \item if $(x,f)\in \tilde T$ is of height $\leq \beta$ so that 
  \begin{enumerate}
   \item $(x_0,f_0)\leq (x,f)$ and $x<x'$,
   \item $q$ is $x'$-compatible with $f$, and
   \item $f(b\uhr \xi)=p_0(b)$ for any $b\in q$ and $\xi\in \eta_{\gamma_\alpha}\cap \htt(f)\setm \gamma_\tau$,\footnote{This notation is a bit unprecise: $p_0$ is not defined on $b\in q$. But $|q|=|p_0|$ so what we mean here is that if $b$ is the $i$th element of $q$ in some canonical enumeration then $f(b\uhr \xi)$ equals the value of $p_0$ on the $i$th element of $\dom p_0$.}
     \end{enumerate}then $(x,f)$ does not extend any element of $\dot X\uhr \alpha$. 
  
 \end{enumerate}

\begin{figure}[H]
  \centering
  \includegraphics[width=0.5\textwidth]{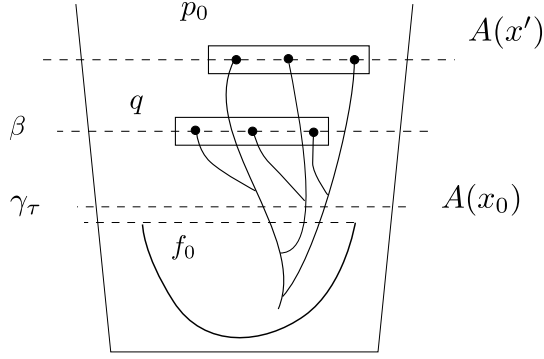}
  \label{diag:bad}
  \caption{The position of bad $q$'s}
  \end{figure}
 
 Our goal is first to show that there are a lot of bad $q$. Observe that any $q$ of the form $(\cup \dom p_0)\uhr \beta$ is bad where $\gamma_\tau\leq \beta<\alpha$. Indeed, if $q=(\cup \dom p_0)\uhr \beta$ is not bad then, since $q$ satisfies (1),(2), there must be some  $(x,f)\in \tilde T$ extending an element of $\dot X\uhr \alpha$ that has properties (3)(a)-(c) as well. This ensures that $(p_0,f_0)\leq (p_0,f)$ in $\mc P_{x'}$ (note how (3)(c) above implies that condition (4) from the definition of $\leq_{\mc P_{x'}}$ holds), and $(p_0,f)\in \mc D$ contradicting our initial assumption on  $(p_0,f_0)$.

 Moreover, to define the set of bad $q$, we used only parameters in $\bar M[\dot C\cap \alpha]$ (e.g. $x',p_0$). So $$\bar M[\dot C\cap \alpha]\models (\cup \dom p_0)\uhr \beta\textmd{ is bad for }\gamma_\tau\leq \beta<\alpha,$$ and hence there is a single $c\in \pi_N(\mc C)$ that forces this. In turn, for any  $\gamma_\tau\leq \beta<\alpha$, $$\bar M\models c\force_{\pi_N(\mc C)} (\cup \dom p_0)\uhr \beta \textmd{ is bad}.$$
 
 So, as $\bar N\prec \bar M$, this must hold in $\bar N[\dot x']\prec \bar M[\dot x']=\bar M$ as well:
 
 $$\bar N[\dot x']\models c\force_{\pi_N(\mc C)} (\cup \dom p_0)\uhr \beta \textmd{ is bad}.$$
 
 In turn, in $\bar N[\dot x']$, the tree $$S_0=\{q\in (\dot A\uhr \alpha)^{|p_0|}:c\force q\textmd{ is bad}\}$$ is uncountable.
 \medskip
 
 Let $S\subs S_0$ be the set of those $q\in S_0$ which have uncountably many extensions in $S_0$. Since ($\bar N[\dot x']$ thinks) $S$  is an uncountable subset of the Aronszajn tree $(\dot A\uhr \alpha)^{|p_0|}$, the next claim follows:
 
 \begin{clm}\label{clm:ext}
  There is a club $B\subs \alpha$ in $\bar N[\dot x']$ so that
  \begin{enumerate}
   \item\label{it:distr}   for any $\eta<\beta\in B$ and $q\in S_\eta$, the set $\{q'\in S_\beta:q\leq q'\}$ contains infinitely many pairwise disjoint elements,
     \item\label{it:guess} $\beta \in B$ implies that $\dot X\uhr \beta,S\uhr \beta\in N_\beta$, and
  \item\label{it:ext} for any $\beta\in B$ and $(x,f)\in \tilde T\uhr \beta$ extending $(x_0,f_0)$ with $x<x'$, there is some $(\bar z,\bar g)\in \tilde T\uhr \beta$ extending $(x,f)$ with $z<x'$ so that $(\bar z,\bar g)$ is above some element of $\dot X$.

  \end{enumerate}
 \end{clm}
 \begin{proof} It suffices to show that the set of those $\beta$ that satisfy either of the three assumptions is a club.
 
 The proof for (\ref{it:distr}) is detailed in \cite[Chapter VI, Lemma 7]{devlin2006souslin}, and is a relatively standard fact on Aronszajn trees (so we omit the proof here).
  
  Next, (\ref{it:guess}) was partly proved in Claim \ref{clm:guess}, and $S\uhr \beta\in N_\beta$ follows likewise: in $\bar N[\dot x']$, the sequence $\pi_N(W^*)=W^*\uhr \alpha$ is still a $\diamondsuit^*$ sequence, so there is  club many $ \beta$ so that $S\uhr \beta\in W_\beta\subs N_\beta$.

  Finally, (\ref{it:ext}) is the content of \cite[Chapter IX, Lemma 7]{devlin2006souslin} but let us give some argument here as well. First, note that it suffices to show that the set of those $\beta$ that satisfy (\ref{it:ext}) is unbounded (because any limit point $\beta$ of a sequence $\beta_n$ with (\ref{it:ext}) will also satisfy (\ref{it:ext})). By elementarity, it suffices that we prove $\alpha$ satisfies the assumptions of (\ref{it:ext}).
  
  Recall that $x'$ defines a $\bar N[C\cap \alpha]$-generic branch $\dot x'$ of the Suslin tree $T\uhr \alpha$. If $\alpha$ would fail (\ref{it:ext}) with some $(x,f)$, then we consider the set $$Z(\dot x')=\{(z,g)\in \tilde T\uhr \alpha:z<x', (x,f)\leq (z,g)\}\in\bar N[C\cap \alpha][\dot x'].$$ Now, there must be some $x\leq x^*<x'$ so that $\bar N[C\cap \alpha]\models x^*\force_T$  ``no element of $Z(\dot x')$ is above some element of $\dot X$''. By the (EP), there is some $f^*$ so that $(x,f)\leq (x^*,f^*)\in \tilde T\uhr \alpha$, and we can further extend $(x^*,f^*)$ to some $(x^{**},f^{**})\in \tilde T\uhr \alpha$ that is above some element of $\dot X\uhr \alpha$ (since $\dot X\uhr \alpha$ is a maximal antichain in $\tilde T\uhr \alpha$). But now $x^{**}\force_T$ ``$(x^{**},f^{**})\in Z(\dot x')$ and is above some element of $\dot X$'', a contradiction.
   
 \end{proof}

%
%
 
 Since $B\in \bar N[\dot x']$ is a club, when we pass to the extension $\bar N[\dot x'][C\cap \alpha]$ there is some end segment of $C\cap \alpha$ contained in $B$ i.e., we can pick some $\delta>\gamma_\tau=\htt (f_0)$ in $\alpha$ so that  $C\cap \alpha\setm \delta\subs B$.
 
 
 \medskip
 
 We are in the final stretch of the proof: find $\gamma_{\tau_1}<\gamma_{\tau_2}\in C\cap \alpha\setm \delta$ so that $\max (\gamma_{\tau_2}\cap \eta_{\gamma_\alpha})<\gamma_{\tau_1}$. This can be done by picking first a limit point $\gamma_{\tau_2}$ of $C\cap \alpha\setm \delta$ and then finding a large enough $\gamma_{\tau_1}<\gamma_{\tau_2}$ in $C\cap \alpha\setm \delta$.
 
 Fix any $q_0\in S_{\gamma_{\tau_1}}$, and let $\psi$ code the values of $p_0$ on  the finite set $\{b\uhr \xi:b\in q_0,\xi\in \eta_{\gamma_\alpha}\cap \gamma_{\tau_1}\setm \gamma_\tau\}$. Use the (EP) to find $(x,f)\in \tilde T_{{\tau_1}}$ (so $\htt (f)=\gamma_{\tau_1}$) such that
 \begin{enumerate}
  \item $(x_0,f_0)\leq (x,f)$ and $x<x'$, and
  \item  $q_0$ is $x'$-compatible with $f$ and $\psi\subs f$.\footnote{The role of $\psi$ is to ensure condition (3)(c) from the definition of bad $q$s.}
 \end{enumerate}

 
 Now, we can extend $(x,f)$ further  to $(\bar z,\bar g)\in \tilde T\uhr {\tau_2}$ above some element of $\dot X$ (since $\gamma_{\tau_2}\in B$ and Claim \ref{clm:ext}(\ref{it:ext})). By a trivial application of the (EP), we may assume that $(\bar z,\bar g)\in \tilde T_{\tau_2}$ i.e., $\htt (\bar g)=\gamma_{\tau_2}$ (to avoid introducing a new notation for $\htt (\bar g)$).

 \medskip

 
 The next final claim will yield the desired contradiction:
 
 \begin{clm}
  There is some $q\in S_{\gamma_{{\tau_2}}}$ such that $q$ is $x'$-compatible with $\bar g$.
 \end{clm}
\begin{proof}

 We prove by induction on $\beta\geq \gamma_{\tau_1}$ (where $\beta\in C\cap \alpha$), that for any extension $(z,g)\in \tilde T$ of $(x,f)$ with $\htt (g)=\beta$, some $q\in S_\beta$ is $x'$-compatible with $g$.
 
 First, if $\beta=\gamma_{\tau_1}$ then $(z,g)=(x,f)$ and  $q=q_0$ works. In general, and very informally, what happens is that all the extensions $(z,g)\in \tilde T$ of $(x,f)$ are defined by taking generic filters for models that contain the levels of $S$, while these levels of $S$ contain infinitely many pairwise disjoint $q$. In turn, some $q$ must be compatible with the generic map $g$; the details follow below.
 
 In the successor step, we are given $(z',g')\in \tilde T$ where $\beta<\beta^+=\htt (g')$ are successive elements of $C\cap \alpha$. Take $(z,g)\in \tilde T$ below $(z',g')$ of height $\beta$, and, using the inductive assumption, find $q\in S_{\beta}$ that is $x'$-compatible with $g$. Recall how $g'$ was constructed: (1) we took an $N_{\beta^+}$-generic filter $H$ for the poset $\mc P_{z',g}$ which gave a function $p^H$ defined on some branches from $B(z')$, and then (2) we found $g'$ defined on $\cup (\dom p^H)$ in such a way that $g'(b\uhr \xi)=p^H(b)$ for all $b\in \dom p^H$ and almost all $\xi\in \eta_{\beta^+}$. 
 
 So, it suffices to show that there is some $q'\in S_{\beta^+}$ so that $q'\subs\dom p^H$ and $p^H\uhr q'$ is constant $h(x')(\beta)$. Such a $q'$ will be $x'$-compatible with $g'$. However, the fact that there are infinitely many pairwise disjoint extensions $q'\in S_{\beta^+}$ of $q$ implies that  the set $$\mc E=\{p\in \mc P_{z',g}:(\exists q'\in S_{\beta^+})\;q\leq q'\subseteq \dom p,\; p\uhr q'\equiv h(x')(\beta^+)\}$$ is dense in $\mc P_{z',g}$ and contained in $N_{\beta^+}$. In turn, $H\cap \mc E\neq \emptyset$ and we are done.
 
 
Finally, suppose that  $(z',g')\in \tilde T$ with $\htt (g')=\beta$ for some limit $\beta$. Again, recall how $g'$ was constructed using an $N_{\beta}$-generic filter for the poset $\mc P_{z'}$, in which we worked with conditions of the form $(p,g)$ where $p$ is a finite function from branches in $B(z')$, $(z,g)\in \tilde T$ for some $z<z'$, and  $b\uhr \htt (g)\subs \dom g$ for any $b\in \dom p$.

In order to show that the generic function $g'$ satisfies our property, we need that the set $$\{(p,g)\in \mc P_{z'}:(\exists q'\in S_{\beta})\;q'\subseteq \dom p,\; p\uhr q'\equiv h(x')(\beta^+)\}$$ is dense in $\mc P_{z'}$. So fix an arbitrary  $(p_1,g_1)\in \mc P_{z'}$ of height $\beta_1$ (without loss of generality, $(z_1,g_1)\geq (x,f)$ with $z_1<z'$). First, by the inductive hypothesis, we can find $q_1\in S_{{\beta_1}}$ that is $x'$-compatible with $g_1$. Then, we can go one level higher to height $\beta_1^+=\min C\setm (\beta_1+1)$, and find $q_2\in S_{{\beta_1^+}}$ above $q_1$ that is disjoint from  $(\cup \dom p_1)\uhr {\beta_1^+}$.

 \begin{figure}[H] 
 \centering

 \includegraphics[width=0.5\textwidth]{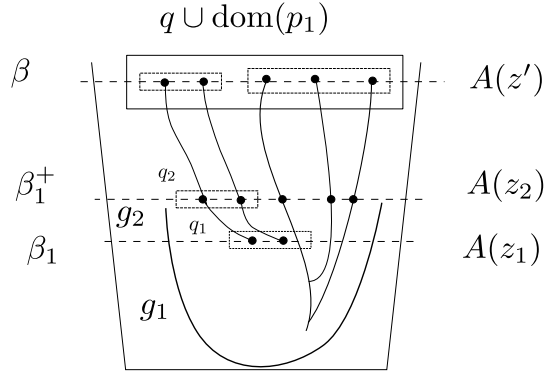}
    \caption{Finding $q$ compatible with $g$}
  \label{diag:final}
   \end{figure}

We can use the (EP) to find $(z_2,g_2)\in \tilde T$ with $\htt( g_2)=\beta_1^+$ extending $(z_1,g_1)$ so that $z_2<z$ and 
\begin{enumerate}[(i)]
 \item $q_2$ is $x'$-compatible with $g_2$, and
 \item $g_2(b\uhr \xi)=p_1(b)$ for all $b\in \dom p_1$ and $\xi\in \eta_{\beta}\cap {\beta_1^+}\setm {\beta_1}$.\footnote{To understand this requirement, just see how the extension in $\mc P_{z'}$ is defined.}
\end{enumerate}
Now pick any $q\in S_{\beta}$ that extends $q_2$, and define $p_2:(\dom p_1)\cup q\to 2$ to be an arbitrary extension of $p_1$. Note that $(p_2,g_2)\in \mc P_{z'}$ is an extension of $(p_1,g_1)$ so that $q$ is $x'$-compatible with $g_2$; see Figure \ref{diag:final} for a summary.
  \end{proof}

 The latter claim clearly contradicts that $q$ was bad since $(\bar z,\bar g)$ satisfies the conditions (3)(a)-(c) from the definition of why $q$ was bad while $(\bar z,\bar g)$ extends some element of $\dot X$.  This, in turn, finishes the proof of the Main Claim \ref{lm:density}.
 
\end{proof}

 We ended the proof now that shows that $\dot X\uhr \alpha$ is a maximal antichain in $\tilde T$, and so $\dot X=\dot X\uhr \alpha$ is countable. In turn, $$V[C][\dot R']\models \tilde T\textmd{ is Suslin}$$
  and so $\tilde T$ must preserve each derived tree $R'\in \partial R$ Suslin. In other words, $R$ remains full Suslin after forcing with $\tilde T$.
  
  \medskip
  
  This finishes the proof of the Main Theorem.

\section{Closing remarks}

Once the reader is familiar with the above proof, it should take no time to realize that one can actually prove the following:

\begin{thm}
 Consistently, CH holds, there exists a full Suslin tree $R$, and for any Aronszajn tree $A$, either
 \begin{enumerate}
  \item a derived tree of $R$ club-embeds into $A$, or
  \item for any ladder system $\underline \eta$, any $\omega$-colouring of $\underline \eta$ has an $A$-uniformization.
 \end{enumerate}

\end{thm}

The point being that so far we aimed to uniformize only constant 2-colourings on a fixed ladder system, and now we are allowed to use $\omega$ colours and vary the ladder system. We decided to prove the special case only (which was enough to yield our corollary on minimal linear orders and Suslin trees) to simplify notation and since we had no further application for this stronger result in mind.

We can also mix-in the Abraham-Shelah forcings from \cite{abrahamiso} with our iteration to achieve that there is a single special Aronszajn tree $U$ so that any Aronszajn tree $A$  that fails (1), actually embeds into $U$ on a club. In this case, any Suslin tree in the resulting model is a countable union of derived trees of $R$ when restricted to an appropriate club.

\medskip

Second, given all the developments in countable support iteration and preservation theorems since the 1980s, we should address why we chose Jensen's iteration framework to prove our result. Especially so that Tadatoshi Miyamoto \cite{spres} proved that a countable support iteration of proper forcings will preserve a Suslin tree $R$ given that each successor step preserves $R$  (also see \cite{heikepres}). So what does prevent us from repeating Moore's countable support iteration to force the uniformizations for certain Aronszajn trees while preserving some fixed tree $R$ being Suslin?\footnote{It would suffice to prove that a single step from the iteration keeps $R$ full Suslin by the Miyamoto preservation theorem. The sufficient conditions detailed in \cite{heikepres} for example (which apply to e.g. the Sacks forcing, or certain weak uniformization forcings) does not apply to our case unfortunately.}

We are unsure at this point if this can be done, but Jensen's framework certainly has a great advantage over everyday CS-iterations, even ones adding no new reals: being ccc, we can guess a countable object in the extension (e.g., the set of bounded branches in $\dot A\uhr \alpha$) by countably many ground model sets (in our case, the sets $B(x)$ collected these branches), and then find a single function that is a uniformization no matter which one of the guesses turns out to be true in the extension. Although, this is not the case for a general CSI of proper forcings, the proof we presented certainly resembles some completeness arguments ubiquitous in no-new-real iterations,\footnote{In particular, see the comments on p. 2686 and 2688 on ``countably many guesses'' \cite{eisworth1999ch}.} and the fusion argument in Moore's work \cite{justinmin}. Despite the similarities, we have not succeeded so far in working out our result using a countable support iteration.

\begin{quest}
 Given a full Suslin tree $R$ and Aronszajn-tree $A$ so that no derived tree of $R$ club embeds into $A$, is there a proper forcing $\mc Q$ which preserves $R$ full Suslin and introduces an $A$-uniformization to a given ladder system colouring?
\end{quest}

It could very well be that Moore's original forcing from \cite{justinmin} does the trick.

\medskip

At this point, we are unsure whether a model like Moore's, where all colourings are uniformized on all Aronszajn trees, can contain any Suslin trees (say if one starts forcing from $L$).

\begin{quest}Is it consistent that there is a Suslin-tree $T$ and ladder system $\underline \eta$ so that any 2-colouring of $\underline \eta$ has a $T$-uniformization?
\end{quest}

\medskip

Moreover, a question of Baumgartner still remains open from \cite{baumorder}:

\begin{quest}

Does $\diamondsuit$ imply the existence of a minimal Aronszajn order?\footnote{Keep in mind that $\diamondsuit^+$ does imply the existence of a minimal Aronszajn order \cite{baumorder,stronglysurj}.}
 
\end{quest}

A possibly relevant result is that $\diamondsuit$ is consistent with the uniformization property restricted to a stationary, costationary $S\subs \omg$ (see \cite[Theorem 1.1, 2.4]{whitehead1}). Let us also refer the reader to Moore's survey \cite{moore2005structural} for a great overview on Aronszajn trees.

\medskip

Finally, there is a quite reasonable but rather different approach to find a model with a Suslin tree but no minimal uncountable linear orders other than $\pm \omg$: start with Moore's model and add a single Cohen real (this was recommended by Moore and Sy Friedman, independently, in personal communication). CH will still  hold, and we do have a Suslin tree in the resulting model \cite{walks}, but it is not hard to see that  the uniformization property for trees will fail. At this point, we could not verify that the extension has no minimal Aronszajn types. 

Yet another angle would be to look at the Sacks model which does have Suslin trees but no minimal real order types; this can be deduced from certain parametrized diamonds \cite{pardiamond}. The same weak diamonds, however, imply that the uniformization property on trees will fail. In any case, if there are no minimal Aronszajn types here, then this model would also provide the first example showing that CH is not necessary to achieve Moore's result.

\end{document}